\documentclass[12pt]{article}
\pdfoutput=1
\RequirePackage[OT1]{fontenc}
\RequirePackage{amsthm,amsmath,amssymb}
\usepackage{libertine}
\usepackage[libertine]{newtxmath}
\RequirePackage{natbib}
\RequirePackage[colorlinks,citecolor=blue,urlcolor=blue]{hyperref}
\usepackage{authblk}
\usepackage{bbm}
\usepackage{multirow}
\usepackage{paralist}
\usepackage[left=1.375in, right=1.375in]{geometry}
\usepackage{graphicx}
\usepackage{mathtools}
\usepackage[bf]{caption}
\usepackage{rotating}
\mathtoolsset{mathic=true}
\usepackage{setspace}
\usepackage[inline]{enumitem}
\usepackage{booktabs}

\newtheorem{Theorem}{Theorem}
\newtheorem{Lemma}{Lemma}
\newtheorem{Proposition}{Proposition}
\theoremstyle{definition}
\newtheorem{Definition}{Definition}

\begin{document}

\title{\textbf{Least favorability of the uniform distribution for tests of the concavity of a distribution function}}
\author{Brendan K.\ Beare\thanks{I thank Zheng Fang for his helpful comments on an early draft of this note.}}
\affil{School of Economics, University of Sydney}
\maketitle
\begin{center}
Accepted for publication in \emph{Stat}.
\end{center}
\medskip
\begin{abstract}
A test of the concavity of a distribution function with support contained in the unit interval may be based on a statistic constructed from the \(L^p\)-norm of the difference between an empirical distribution function and its least concave majorant. It is shown here that the uniform distribution is least favorable for such a test, in the sense that the limiting distribution of the statistic obtained under uniformity stochastically dominates the limiting distribution obtained under any other concave distribution function.
\end{abstract}
	
	\newpage
	
\section{Introduction}
Denote by \(\mathcal F\) the collection of distribution functions \(F:[0,\infty)\to[0,1]\) which are continuous, assign zero mass at zero, and have support contained in the unit interval, so that \(F(0)=0\) and \(F(x)=1\) for \(x\geq1\).
Let \(X_1,\dots,X_n\) be an independent and identically distributed collection of random variables with common distribution function \(F\in\mathcal F\) and empirical distribution function \(F_n:[0,\infty)\to[0,1]\). We are concerned with tests of the null hypothesis that \(F\) is concave, with test statistics based upon the distance between \(F_n\) and its least concave majorant (LCM). Specifically, we consider statistics of the form
\begin{equation}
	S_{n,p}=S_{n,p}(X_1,\dots,X_n)=\sqrt{n}\Vert\mathcal MF_n-F_n\Vert_p,
\end{equation}
where \(\Vert\cdot\Vert_p\) is the usual \(L^p\)-norm with \(p\in[1,\infty]\), and \(\mathcal MF_n\) is the LCM of \(F_n\), defined carefully in the following section. If \(F\) is concave, then the statistic \(S_{n,p}\) has a weak limit (i.e., limiting distribution) as \(n\to\infty\). The contribution of this note is to show that, among all the weak limits of \(S_{n,p}\) corresponding to different concave choices of \(F\in\mathcal F\), there is one weak limit which first-order stochastically dominates all of the others, and that weak limit is obtained when \(F\) is the uniform distribution (on the unit interval). In this sense, the uniform distribution is least favorable under the null.

The problem we study has been considered in earlier work. \citet{KL2004,KL2008} showed that the uniform distribution is least favorable when \(p=\infty\), even for fixed \(n\), but wrote of \(S_{n,p}\) for the case \(p\in[1,\infty)\) that ``its limiting distribution under uniformity can be obtained, but we cannot prove that the uniform distribution is least favorable'' \citep[p.\ 5]{KL2004}. Instead, they showed that the uniform distribution is least favorable for the modified quantities
\begin{align}
	R_{n,p}&=\sqrt{n}\left(\int_0^1(\mathcal MF_n(u)-F_n(u))^p\mathrm{d}F(u)\right)^{1/p},\\ T_{n,p}&=\sqrt{n}\left(\int_0^1(\mathcal MF_n(u)-F_n(u))^p\mathrm{d}F_n(u)\right)^{1/p},
\end{align}
which are based on weighted, rather than unweighted, integrals. They achieved this by showing that the quantities \(S_{n,\infty}\), \(R_{n,p}\) and \(T_{n,p}\) are guaranteed to weakly increase if \(F\) is concave and our observations \(X_1,\dots,X_n\) are replaced with the uniformly distributed random variables \(F(X_1),\dots,F(X_n)\). That is,
\begin{equation}
	S_{n,\infty}(X_1,\dots,X_n)\leq S_{n,\infty}(F(X_1),\dots,F(X_n)),
\end{equation}
and similarly for \(R_{n,p}\) and \(T_{n,p}\). This is Proposition 3.1 of \citet{KL2008}. Closely related results were obtained by \citet{C2002}, and by \citet{D2003} in a monotone regression context.

As noted by \citet[p.\ 361]{KL2008}, for \(p\in[1,\infty)\) it may not be true that the statistic \(S_{n,p}\) must weakly increase if \(F\) is concave and we replace our observations \(X_1,\dots,X_n\) with \(F(X_1),\dots,F(X_n)\). For instance, for a sample of size \(n=2\) with observations \(X_1=0.25\) and \(X_2=1\), it is easy to verify that \(S_{2,2}(0.25,1)=0.37\) (rounding to two decimal places). If \(F(u)=\sqrt{u}\) and we instead use as our observations \(F(X_1)=0.5\) and \(F(X_2)=1\), then we compute \(S_{2,2}(0.5,1)=0.29\), which is smaller than the statistic computed from the original observations. This example shows that it is possible that \(S_{n,2}(X_1,\dots,X_n)>S_{n,2}(F(X_1),\dots,F(X_n))\) when \(F\) is concave. The main result of this note shows that it is nevertheless always the case that, for any \(p\in[1,\infty]\), the weak limit of \(S_{n,p}(X_1,\dots,X_n)\) is first-order stochastically dominated by the weak limit of \(S_{n,p}(F(X_1),\dots,F(X_n))\) when \(F\) is concave.

The least favorability of the uniform distribution may be unsurprising. Uniformity is the unique case where \(F\) is linear on the unit interval, and, at an intuitive level, we may regard the linear functions to be the subset of the concave functions which are as close as possible to violating concavity. However, this intuition has proved to be misleading in the closely related context of testing whether an ordinal dominance curve (ODC), also called a receiver operating characteristic curve, is concave. \citet{CT2005} proposed a test of the concavity of an ODC based on a statistic constructed from the \(L^p\)-norm of the difference between an empirical analogue to the ODC, and its LCM. \citet{BM2015} proved the following surprising fact about this test: the uniform distribution is least favorable if \(p\in[1,2]\), but is not least favorable if \(p\in(2,\infty]\). In fact, for \(p\in(2,\infty]\), one can find sequences of concave ODCs along which the weak limit of the test statistic diverges to infinity. Such sequences exhibit increasingly steep affine segments near zero. The main result of this note shows that this surprising aspect of testing the concavity of an ODC does not manifest in the simpler context of testing the concavity of a distribution function.

The motivation for writing this note is a recent technical report by \citet{EKW2020} in which statistical tests for p-hacking are developed and applied to academic publication data. One of their procedures involves using the statistic \(S_{n,p}\) to test whether the distribution of reported p-values is concave, which ought to be the case in the absence of p-hacking. The validity of their choice of critical value relies on the least favorability of the uniform distribution, for which they appeal to this note, and to \citet{KL2008}.

Before proceeding further we introduce some additional notation and terminology. We denote by \(\mathbf R^+\) the nonnegative half-line \([0,\infty)\). Given a nonempty convex set \(I\subseteq\mathbf R^+\), we denote by \(\ell^\infty(I)\) the collection of all uniformly bounded, real-valued functions on \(I\), and we denote by \(\ell^\infty_I\) the collection of all uniformly bounded, real-valued functions whose domain is any convex set \(J\) with \(I\subseteq J\subseteq\mathbf R^+\). The two collections \(\ell^\infty(I)\) and \(\ell^\infty_I\) are the same when \(I=\mathbf R^+\); we denote this collection by \(\ell^\infty\), and regard it as a real Banach space equipped with the uniform norm. We denote by \(C_0\) the subspace of \(\ell^\infty\) consisting of all continuous functions in \(\ell^\infty\) vanishing at infinity. We denote by \(\rightsquigarrow\) weak convergence in \(\ell^\infty\) in the sense of Hoffman-J\o rgensen. We denote by \(\to_\mathrm{d}\) the convergence in distribution  of a sequence of real-valued random variables. We refer to a centered Gaussian process on \([0,1]\) as a Wiener process if it has covariance function \((u,v)\mapsto u\wedge v\), or as a Brownian bridge if it has covariance function \((u,v)\mapsto u\wedge v-uv\).

\section{Results}

Since the least favorability of the uniform distribution when \(p=\infty\) was shown already by \citet{KL2004,KL2008}, we assume in what follows that \(p\in[1,\infty)\). The following definition is adapted from Definition 2.1 of \citet{BF2017}.
\begin{Definition}
	Given a nonempty convex set \(I\subseteq\mathbf R^+\), the LCM over \(I\) is the mapping \(\mathcal M_I:\ell^\infty_I\to\ell^\infty(I)\) that transforms each \(\theta\in\ell^\infty_I\) to the function
	\begin{equation}
		\mathcal M_I\theta(x)=\inf\{g(x):g\in\ell^\infty(I),\,g\text{ is concave, and }\theta\leq g\text{ on }I\},\quad x\in I.
	\end{equation}
	We write \(\mathcal M\) as shorthand for \(\mathcal M_{\mathbf R^+}\), and refer to \(\mathcal M\) as the LCM operator.
\end{Definition}
Critical to our analysis is the fact that the LCM operator satisfies a property called Hadamard directional differentiability. We define this property as in Definition 2.2 of \citet{BF2017}.
\begin{Definition}
	Let \(\mathbf D\) and \(\mathbf E\) be real Banach spaces. A map \(\phi:\mathbf D\to\mathbf E\) is said to be Hadamard directionally differentiable at \(\theta\in\mathbf D\) tangentially to a set \(\mathbf D_0\subset\mathbf D\) if there is a map \(\phi'_\theta:\mathbf D_0\to\mathbf E\) such that
	\begin{equation}
		\left\Vert\frac{\phi(\theta+t_nh_n)-\phi(\theta)}{t_n}-\phi'_\theta(h)\right\Vert_{\mathbf E}\to0
	\end{equation}
	for all \(h\in\mathbf D_0\) and all \(h_1,h_2,\dots\in\mathbf D\) and \(t_1,t_2,\dots\in(0,\infty)\) such that \(t_n\downarrow0\) and \(\Vert h_n-h\Vert_{\mathbf D}\to0\). The map \(\phi'_\theta\) is called the Hadamard directional derivative of \(\phi\) at \(\theta\) tangentially to \(\mathbf D_0\).
\end{Definition}
Hadamard directional differentiability differs from the usual notion of Hadamard differentiability in that the approximating map \(\phi'_\theta\) is not required to be linear. The following result on the Hadamard directional differentiability of the LCM operator is Proposition 2.1 of \citet{BF2017}. A closely related result was established previously by \citet[Lemma 3.2]{BM2015}. See \citet{BS2019} for a graphical illustration of the Hadamard directional derivative of the LCM operator, and an alternative representation of it in terms of a three-dimensional contact set.
\begin{Lemma}\label{HadamardLem}
	The LCM operator \(\mathcal M:\ell^\infty\to\ell^\infty\) is Hadamard directionally differentiable at any concave \(\theta\in\ell^\infty\) tangentially to \(C_0\). Its Hadamard directional derivative \(\mathcal M'_\theta:C_0\to\ell^\infty\) is uniquely determined as follows: for any \(h\in C_0\) and \(x\in\mathbf R^+\), we have \(\mathcal M'_\theta h(x)=\mathcal M_{\{x\}\cup I_{\theta,x}}h(x)\), where \(I_{\theta,x}\) is the union of all open intervals \(I\subset\mathbf R^+\) such that
	\begin{enumerate*}
		\item[(1)] \(x\in I\), and
		\item[(2)] \(\theta\) is affine on \(I\).
	\end{enumerate*}
\end{Lemma}
With the Hadamard directional derivative of the LCM operator in hand, it is straightforward to deduce the weak limit of our test statistic \(S_{n,p}\) when \(F\) is concave by applying the functional Delta method. This has been done by \citet{F2019}. Specifically, commencing from the weak convergence \(\sqrt{n}(F_n-F)\rightsquigarrow\mathbb B_F=\mathbb B\circ F\) guaranteed by Donsker's theorem, with \(\mathbb B\) a Brownian bridge, we may apply the functional Delta method to obtain \(\sqrt{n}(\mathcal MF_n-F_n)\rightsquigarrow\mathcal M'_F\mathbb B_F-\mathbb B_F\). That the functional Delta method may be applied when the relevant Hadamard directional derivative is nonlinear, as is generally the case with the LCM operator, was shown by \citet{S1991}. An application of the continuous mapping theorem now gives the following result, which is immediate from Lemma 3.1 of \citet{F2019}.
\begin{Lemma}\label{FangLem}
	For any concave \(F\in\mathcal F\), we have \(S_{n,p}\to_\mathrm{d}\Vert\mathcal M'_F\mathbb B_F-\mathbb B_F\Vert_p\), where \(\mathcal M'_F\) is given by Lemma \ref{HadamardLem}.
\end{Lemma}
In order to demonstrate the least favorability of the uniform distribution, it will be helpful to re-express the weak limit in Lemma \ref{FangLem} in terms of a mutually independent collection of Wiener processes. This will require some additional notation. Given a concave \(F\in\mathcal F\) and \(x\in(0,\infty)\), the set \(I_{F,x}\) defined in Lemma \ref{HadamardLem} is an open interval if \(F\) is affine in a neighborhood of \(x\), and is empty otherwise. Consider the collection \(\mathcal K\) of all such open intervals as we vary \(x\) over \((0,\bar{x})\), where \(\bar{x}=\inf\{x>0:F(x)=1\}\). The collection \(\mathcal K\) is nonempty if \(F\) is concave but not strictly concave on \((0,\bar{x})\), or empty if \(F\) is strictly concave on \((0,\bar{x})\). Any two open intervals in \(\mathcal K\) are disjoint by construction, so \(\mathcal K\) is finite or countable. If \(\mathcal K\) is nonempty then we write the intervals in \(\mathcal K\) as \((a_k,b_k)\) with \(k\) ranging over \(K\), a finite or countable index set. For each \(k\in K\), we define \(d_k=b_k-a_k\) (\(d\) for depth) and \(h_k=F(b_k)-F(a_k)\) (\(h\) for height). Note that if \(F\) is the uniform distribution then \(\mathcal K\) has a single element: the open unit interval. Given a function \(\theta\in\ell^\infty([0,1])\), we define \(\mathcal D\theta=\mathcal M_{[0,1]}\theta-\theta\).
\begin{Proposition}\label{SnpWeakLim}
	For any concave \(F\in\mathcal F\), if \(\mathcal K\) is nonempty, then the weak limit of \(S_{n,p}\) given in Lemma \ref{FangLem} satisfies
	\begin{equation}\label{prop1eq}
		\Vert\mathcal M_F'\mathbb B_F-\mathbb B_F\Vert_p=\left(\sum_{k\in K}d_kh_k^{p/2}\Vert\mathcal D \mathbb W_k\Vert^p_p\right)^{1/p},
	\end{equation}
	where \(\{\mathbb W_k:k\in K\}\) is a mutually independent collection of Wiener processes. In particular, if \(F\) is the uniform distribution on \([0,1]\), then \(\Vert\mathcal M_F'\mathbb B_F-\mathbb B_F\Vert_p=\Vert\mathcal D\mathbb W\Vert_p\), where \(\mathbb W\) is a Wiener process. If instead \(\mathcal K\) is empty, then \(\Vert\mathcal M_F'\mathbb B_F-\mathbb B_F\Vert_p=0\).
\end{Proposition}
Proposition \ref{SnpWeakLim} should be compared to Theorem 3.1 of \citet{BM2015}, which establishes the weak limit of a similar statistic used for testing the concavity of an ODC. The former result may be heuristically obtained by setting \(\lambda=1\) in the latter. Note that if \(\mathbb W\) is a Wiener process and \(\mathbb B\) is the Brownian bridge \(\mathbb B(u)=\mathbb W(u)-u\mathbb W(1)\), then \(\mathcal D\mathbb W=\mathcal D\mathbb B\). This follows from the following property of the LCM, which is part of Lemma 2.1 of \citet{DT2003}: for any two functions \(\theta_1,\theta_2\in\ell^\infty([0,1])\) with \(\theta_2\) affine, we have \(\mathcal M_{[0,1]}(\theta_1+\theta_2)=\mathcal M_{[0,1]}\theta_1+\theta_2\). Thus, the collection of mutually independent Wiener processes in Proposition \ref{SnpWeakLim} could equivalently be replaced with a mutually independent collection of Brownian bridges. The expression involving Wiener processes is more convenient in our proofs in the following section, which rely on the self-similarity and independent increments properties of Wiener processes.

Least favorability of the uniform distribution is established by the following result, which is the main contribution of this note.
\begin{Theorem}\label{mainresult}
	For any concave \(F\in\mathcal F\), the weak limit of \(S_{n,p}\) is first-order stochastically dominated by \(\Vert\mathcal D\mathbb W\Vert_p\), where \(\mathbb W\) is a Wiener process.
\end{Theorem}
Theorem \ref{mainresult} is valid for any \(p\in[1,\infty]\). This was shown by \citet{KL2004,KL2008} for the case \(p=\infty\), but is, to the best of our knowledge, a new result for \(p\in[1,\infty)\). It is interesting to compare Theorem \ref{mainresult} to Theorems 4.1 and 4.2 of \citet{BM2015}. The latter results show, in a similar context involving testing the concavity of an ODC, that the least favorability of the uniform distribution hinges critically on having \(p\leq2\). Theorem \ref{mainresult} shows that the simpler test of concavity studied here is better behaved, with the uniform distribution least favorable for all \(p\in[1,\infty]\).

For the reader's convenience, in Table \ref{criticalvalues} we report approximate \((1-\alpha)\)-quantiles of \(\Vert\mathcal D\mathbb W\Vert_p\) for \(\alpha=0.01,0.05,0.1\) and \(p=1,2,\infty\). These values were obtained by numerical simulation. They can be used as critical values for tests of the concavity of a distribution function based on the statistic \(S_{n,p}\). The values in the final column of Table \ref{criticalvalues} corresponding to \(p=\infty\) are close to those in Table 1 of \citet{D2003} and in Table 1 of \citet{KL2004}.
\begin{table}
	\caption{Simulated \((1-\alpha)\)-quantiles of \(\Vert\mathcal D\mathbb W\Vert_p\).}
	\centering
	\begin{tabular}{rccc}
		\toprule
		&\(p=1\)	& \(p=2\)	& \(p=\infty\)\\
		\midrule
		\(\alpha=0.01\)& 0.80		& 0.91			& 1.68\\
		\(\alpha=0.05\)& 0.65		& 0.74			& 1.43\\
		\(\alpha=0.10\)& 0.57		& 0.66			& 1.30\\
		\bottomrule
	\end{tabular}
	\label{criticalvalues}
\end{table}

\section{Proofs}

\begin{proof}[Proof of Lemma \ref{HadamardLem}]
	This is Proposition 2.1 of \citet{BF2017}.
\end{proof}

\begin{proof}[Proof of Lemma \ref{FangLem}]
	Lemma \ref{FangLem} can be deduced from Lemma \ref{HadamardLem} by arguing as in the paragraph immediately following the statement of the latter result. It is the same as Lemma 3.1 of \citet{F2019}, but with the additional requirement that \(F(1)=1\), which is not needed.
\end{proof}

\begin{proof}[Proof of Proposition \ref{SnpWeakLim}]
	Suppose first that \(\mathcal K\) is empty. In this case \(F\) is strictly concave on \([0,\bar{x}]\) and affine on \((\bar{x},\infty)\), and we deduce from Lemma \ref{HadamardLem} that
	\begin{equation}
		\mathcal M'_F\mathbb B_F(x)=\begin{cases}\mathbb B_F(x)&\text{if }0\leq x\leq\bar{x}\\
			\mathcal M_{(\bar{x},\infty)}\mathbb B_F(x)&\text{if }x>\bar{x}.\end{cases}
	\end{equation}
	Since \(\mathbb B_F(x)=0\) for all \(x>\bar{x}\), we have \(\mathcal M_{(\bar{x},\infty)}\mathbb B_F(x)=\mathbb B_F(x)=0\) for all \(x>\bar{x}\). We therefore have \(\mathcal M'_F\mathbb B_F(x)=\mathbb B_F(x)\) for all \(x\geq0\), and hence \(\Vert\mathcal M'_F\mathbb B_F-\mathbb B_F\Vert_p=0\), as claimed.
	
	Suppose instead that \(\mathcal K\) is not empty. Let \(\mathbb W\) be a Wiener process such that \(\mathbb B(u)=\mathbb W(u)-u\mathbb W(1)\). For \(k\in K\) and \(u\in[0,1]\), define
	\begin{equation}
		\mathbb W_k(u)=h_k^{-1/2}(\mathbb W(F(a_k)+h_ku)-\mathbb W(F(a_k)).
	\end{equation}
	The self-similarity property of the Wiener process \(\mathbb W\) implies that each \(\mathbb W_k\) is a Wiener process, while the independent increments property of \(\mathbb W\) implies that the collection of Wiener processes \(\{\mathbb W_k,k\in K\}\) is mutually independent. Let \(\mathbb W_F=\mathbb W\circ F\), and note that, for \(k\in K\) and \(u\in[0,1]\), since \(F(a_k+d_ku)=F(a_k)+h_ku\), we have
	\begin{equation}\label{eq:WF}
		\mathbb W_k(u)=h_k^{-1/2}(\mathbb W_F(a_k+d_ku)-\mathbb W_F(a_k)).
	\end{equation}
	Lemma 2.1 of \citet{DT2003} implies that, for \(k\in K\) and \(u\in[0,1]\),
	\begin{equation}\label{eq:DT1}
		\mathcal M_{[0,1]}\mathbb W_k(u)=h_k^{-1/2}([\mathcal M_{[a_k,b_k]}\mathbb W_F](a_k+d_ku)-\mathbb W_F(a_k)),
	\end{equation}
	where we enclose \(\mathcal M_{[a_k,b_k]}\mathbb W_F\) in square brackets to emphasize that we are taking the LCM of \(\mathbb W_F(x)\) as a function of \(x\), not the LCM of \(\mathbb W_F(a_k+d_ku)\) as a function of \(u\). Subtracting \eqref{eq:WF} from \eqref{eq:DT1}, we obtain
	\begin{equation}\label{eq:difference}
		\mathcal D\mathbb W_k(u)=h_k^{-1/2}([\mathcal M_{[a_k,b_k]}\mathbb W_F](a_k+d_ku)-\mathbb W_F(a_k+d_ku)).
	\end{equation}
	Since \(\mathbb W_F-\mathbb B_F\) is affine on each interval \([a_k,b_k]\), another application of Lemma 2.1 of \citet{DT2003} shows that we may rewrite \eqref{eq:difference} as
	\begin{equation}\label{eq:key1}
		\mathcal D\mathbb W_k(u)=h_k^{-1/2}([\mathcal M_{[a_k,b_k]}\mathbb B_F](a_k+d_ku)-\mathbb B_F(a_k+d_ku)).
	\end{equation}
	Raising both sides of \eqref{eq:key1} to the power of \(p\), integrating over \(u\), and applying the change-of-variables \(v=a_k+d_ku\), we obtain
	\begin{align}
		\Vert\mathcal D\mathbb W_k\Vert_p^p&=h_k^{-p/2}\int_0^1([\mathcal M_{[a_k,b_k]}\mathbb B_F](a_k+d_ku)-\mathbb B_F(a_k+d_ku))^p\mathrm{d}u\\
		&=d_k^{-1}h_k^{-p/2}\int_{a_k}^{b_k}(M_{[a_k,b_k]}\mathbb B_F(v)-\mathbb B_F(v))^p\mathrm{d}v.
	\end{align}
	Multiplying by \(d_kh_k^{p/2}\) and summing over \(k\in K\) gives
	\begin{align}
		\sum_{k\in K}d_kh_k^{p/2}\Vert\mathcal D\mathbb W_k\Vert_p^p&=\sum_{k\in K}\int_{a_k}^{b_k}(M_{[a_k,b_k]}\mathbb B_F(v)-\mathbb B_F(v))^p\mathrm{d}v\\&=\int_0^1(\mathcal M'_F\mathbb B_F(v)-\mathbb B_F(v))^p\mathrm{d}v,
	\end{align}
	yielding \eqref{prop1eq} as claimed. In the particular case where \(F\) is the uniform distribution on \([0,1]\), it is apparent that the sole interval in \(\mathcal K\) is the open unit interval. Therefore, setting \(K=\{1\}\), we have \(d_1=1\), \(h_1=1\), and \(\Vert\mathcal M'_F\mathbb B_F-\mathbb B_F\Vert_p=\Vert\mathcal D\mathbb W_1\Vert_p\).
\end{proof}

\begin{proof}[Proof of Theorem \ref{mainresult}]
	If \(\mathcal K\) is empty then the result is immediate from Proposition \ref{SnpWeakLim}, so suppose that \(\mathcal K\) is nonempty. For \(k\in K\), define
	\begin{equation}
		l_k=d_k^{2/(p+2)}h_k^{p/(p+2)}.
	\end{equation}
	Let \(\mathbb W\) be a Wiener process. In view of Proposition \ref{SnpWeakLim}, it suffices for us to find \(\{\mathbb W_k,k\in K\}\), a mutually independent collection of Wiener processes, such that
	\begin{equation}\label{objective}
		\left(\sum_{k\in K}l_k^{(p+2)/2}\Vert\mathcal D\mathbb W_k\Vert_p^p\right)^{1/p}\leq\Vert\mathcal D\mathbb W\Vert_p.
	\end{equation}
	The well-known inequality between weighted geometric and arithmetic means implies that
	\begin{equation}
		l_k\leq\frac{2}{p+2}d_k+\frac{p}{p+2}h_k.
	\end{equation}
	Consequently, since \(\sum_{k}d_k\leq1\) and \(\sum_{k}h_k\leq1\) (because \(F(1)=1\)), we have \(\sum_kl_k\leq1\). We thus deduce the existence of a disjoint collection of open intervals \((a_k^\ast,b_k^\ast)\subset[0,1]\), \(k\in K\), such that \(b_k^\ast-a_k^\ast=l_k\) for each \(k\in K\). For \(k\in K\) and \(u\in[0,1]\), define
	\begin{equation}\label{T1eq1}
		\mathbb W_k(u)=l_k^{-1/2}(\mathbb W(a_k^\ast+l_ku)-\mathbb W(a_k^\ast)).
	\end{equation}
	The self-similarity property of the Wiener process \(\mathbb W\) implies that each \(\mathbb W_k\) is a Wiener process, while the independent increments property of \(\mathbb W\) implies that the collection of Wiener processes \(\{\mathbb W_k,k\in K\}\) is mutually independent. Lemma 2.1 of \citet{DT2003} implies that, for \(k\in K\) and \(u\in[0,1]\),
	\begin{equation}\label{T1eq2}
		\mathcal M_{[0,1]}\mathbb W_k(u)=l_k^{-1/2}([\mathcal M_{[a_k^\ast,b_k^\ast]}\mathbb W](a_k^\ast+l_ku)-\mathbb W(a_k^\ast)),
	\end{equation}
	and so, by subtracting \eqref{T1eq1} from \eqref{T1eq2}, we obtain
	\begin{equation}\label{T1eq3}
		\mathcal D\mathbb W_k(u)=l_k^{-1/2}([\mathcal M_{[a_k^\ast,b_k^\ast]}\mathbb W](a_k^\ast+l_ku)-\mathbb W(a_k^\ast+l_ku)).
	\end{equation}
	Raising both sides of \eqref{T1eq3} to the power of \(p\), integrating over \(u\), and applying the change-of-variables \(v=a_k^\ast+l_ku\), we obtain
	\begin{align}
		\Vert\mathcal D\mathbb W_k\Vert_p^p&=l_k^{-p/2}\int_0^1([\mathcal M_{[a_k^\ast,b_k^\ast]}\mathbb W](a_k^\ast+l_ku)-\mathbb W(a_k^\ast+l_ku))^p\mathrm{d}u\\
		&=l_k^{-(p+2)/2}\int_{a_k^\ast}^{b_k^\ast}(\mathcal M_{[a_k^\ast,b_k^\ast]}\mathbb W(v)-\mathbb W(v))^p\mathrm{d}v.
	\end{align}
	Multiplying by \(l_k^{(p+2)/2}\) and summing over \(k\in K\) gives
	\begin{align}
		\sum_{k\in K}l_k^{(p+2)/2}\Vert\mathcal D\mathbb W_k\Vert_p^p&=\sum_{k\in K}\int_{a_k^\ast}^{b_k^\ast}(\mathcal M_{[a_k^\ast,b_k^\ast]}\mathbb W(v)-\mathbb W(v))^p\mathrm{d}v\\&\leq\sum_{k\in K}\int_{a_k^\ast}^{b_k^\ast}\mathcal D\mathbb W(v)^p\mathrm{d}v\leq\int_0^1\mathcal D\mathbb W(v)^p\mathrm{d}v,
	\end{align}
	and \eqref{objective} follows.
\end{proof}

\end{document}